\documentclass[platex]{amsart}
\usepackage{amssymb,amsmath,amsthm,empheq,url}

\usepackage[top=30truemm,bottom=30truemm,left=20truemm,right=20truemm]{geometry}

\theoremstyle{definition}
\newtheorem{theorem}{Theorem}[section]
\newtheorem{proposition}[theorem]{Proposition}

\newtheorem{lemma}[theorem]{Lemma}
\newtheorem{corollary}[theorem]{Corollary}
\newtheorem{remark}[theorem]{Remark}
\newtheorem*{remark*}{Remark}

\DeclareMathOperator{\re}{Re}
\DeclareMathOperator{\im}{Im}

\DeclareMathOperator{\Mod}{Mod}

\DeclareMathOperator{\app}{app}

\title[Minimal mass blow-up solutions for double power NLS with an inverse power potential]{Minimal mass blow-up solutions for double power nonlinear Schr\"{o}dinger equations with an inverse power potential}
\author[N. Matsui]{Naoki Matsui}
\date{\today}
\address[N. Mastui]{Department of Mathematics\\ Tokyo University of Science\\ 1-3 Kagurazaka, Shinjuku-ku, Tokyo 162-8601, Japan}
\email[N. Matsui]{1120703@ed.tus.ac.jp}


\begin{document}
\maketitle

\begin{abstract}
We consider the following nonlinear Schr\"{o}dinger equation with double power nonlinearities and an inverse power potential:
\[
i\frac{\partial u}{\partial t}+\Delta u+|u|^{\frac{4}{N}}u+C_1|u|^{p-1}u+\frac{C_2}{|x|^{2\sigma}}u=0
\]
in $\mathbb{R}^N$. From the classical argument, the solution with subcritical mass ($\left\|u_0\right\|_2<\left\|Q\right\|_2$) is global and bounded in $H^1(\mathbb{R}^N)$, where $Q$ is the ground state of the mass-critical problem. Previous results show the existence of a minimal-mass blow-up solution for the equation with $C_1>0$ and $C_2=0$ or $C_1=0$ and $C_2>0$ and investigate the behaviour of the solution near the blow-up time. Moreover, they have suggested that a subcritical power nonlinearity and an inverse power potential behave in a similar way with respect to blow-up. On the other hand, the previous results also show the nonexistence of a minimal-mass blow-up solution for the equation with $C_1<0$ and $C_2=0$ or $C_1=0$ and $C_2<0$. In this paper, we investigate the existence and behaviour of a minimal-mass blow-up solution for the equation with $C_1>0>C_2$ or $C_1<0<C_2$, that is the subcritical power nonlinearity and the inverse power potential cancel each other's effects. Furthermore, we give a lower estimate of the arbitrary finite-time blow-up solution with critical mass and show that the energies of critical-mass blow-up solutions are positive when $(C_1,C_2,p,\sigma)$ is under certain conditions.
\end{abstract}

\section{Introduction}
We consider the following nonlinear Schr\"{o}dinger equation with double power nonlinearlities and an inverse power potential:
\begin{align}
\label{NLS}
i\frac{\partial u}{\partial t}+\Delta u+|u|^{\frac{4}{N}}u+C_1|u|^{p-1}u+\frac{C_2}{|x|^{2\sigma}}u=0
\end{align}
in $\mathbb{R}^N$, where $N\in\mathbb{N}$, $p$ and $\sigma$ are positive, $C_1$ and $C_2$ is real, and double-sign do not correspond. It is well known that if
\begin{align}
\label{index1}
1<p<1+\frac{4}{N}\quad\text{and}\quad 0<\sigma<\min\left\{\frac{N}{2},1\right\},
\end{align}
then \eqref{NLS} is locally well-posed in $H^1(\mathbb{R}^N)$ from \cite[Proposition 3.2.2, Proposition 3.2.5, Theorem 3.3.9, and Proposition 4.2.3]{CSSE}. This means that for any initial value $u_0\in H^1(\mathbb{R}^N)$, there exists a unique maximal solution $u\in C((T_*,T^*),H^1(\mathbb{R}^N))\cap C^1((T_*,T^*),H^{-1}(\mathbb{R}^N))$ for \eqref{NLS} with $u(0)=u_0$. Moreover, the mass (i.e., $L^2$-norm) and energy $E$ of the solution $u$  are conserved by the flow, where 
\[
E(u):=\frac{1}{2}\left\|\nabla u\right\|_2^2-\frac{1}{2+\frac{4}{N}}\left\|u\right\|_{2+\frac{4}{N}}^{2+\frac{4}{N}}-\frac{C_1}{p+1}\|u\|_{p+1}^{p+1}-\frac{C_2}{2}\||x|^{-\sigma}u\|_2^2.
\]
Furthermore, the blow-up alternative holds:
\[
T^*<\infty\quad \text{implies}\quad \lim_{t\nearrow T^*}\left\|\nabla u(t)\right\|_2=\infty.
\]

We define $\Sigma^k$ by
\[
\Sigma^k:=\left\{u\in H^k\left(\mathbb{R}^N\right)\ \middle|\ |x|^ku\in L^2\left(\mathbb{R}^N\right)\right\},\quad \|u\|_{\Sigma^k}^2:=\|u\|_{H^k}^2+\||x|^ku\|_2^2.
\]
Particularly, $\Sigma^1$ is called the virial space. If $u_0\in \Sigma^1$, then the solution $u$ for \eqref{NLS} with $u(0)=u_0$ belongs to $C((T_*,T^*),\Sigma^1)$ from \cite[Lemma 6.5.2]{CSSE}.

Moreover, we consider the case
\begin{align}
\label{index2}
1<p<1+\frac{4}{N}\quad\text{and}\quad 0<\sigma<\min\left\{\frac{N}{4},1\right\}.
\end{align}
Under this condition, if $u_0\in H^2(\mathbb{R}^N)$, then the solution $u$ for \eqref{NLS} with $u(0)=u_0$ belongs to $C((T_*,T^*),H^2(\mathbb{R}^N))\cap C^1((T_*,T^*),L^2(\mathbb{R}^N))$ and $|x|\nabla u\in C((T_*,T^*),L^2(\mathbb{R}^N))$ from \cite[Theorem 5.3.1]{CSSE}. Furthermore, if $u_0\in \Sigma^2$, then the solution $u$ for \eqref{NLS} with $u(0)=u_0$ belongs to $C((T_*,T^*),\Sigma^2)\cap C^1((T_*,T^*),L^2(\mathbb{R}^N))$ and $|x|\nabla u\in C((T_*,T^*),L^2(\mathbb{R}^N))$ from the same proof as in \cite[Lemma 6.5.2]{CSSE}.

\subsection{Critical problem}
Firstly, we describe the results regarding the mass-critical problem:
\begin{align}
\label{CNLS}
i\frac{\partial u}{\partial t}+\Delta u+|u|^{\frac{4}{N}}u=0,\quad (t,x)\in\mathbb{R}\times\mathbb{R}^N.
\end{align}
In particular, \eqref{NLS} with $\sigma=0$ and $p=1$ is reduced to \eqref{CNLS}.

It is well known (\cite{BLGS,KGS,WGS}) that there exists a unique classical solution $Q$ for
\[
-\Delta Q+Q-\left|Q\right|^{\frac{4}{N}}Q=0,\quad Q\in H^1(\mathbb{R}^N),\quad Q>0,\quad Q\text{\ is\ radial},
\]
which is called the ground state. If $\|u\|_2=\|Q\|_2$ ($\|u\|_2<\|Q\|_2$, $\|u\|_2>\|Q\|_2$), we say that $u$ has the \textit{critical mass} (\textit{subcritical mass}, \textit{supercritical mass}, respectively).

We note that $E_{\text{crit}}(Q)=0$, where $E_{\text{crit}}$ is the energy with respect to \eqref{CNLS}. Moreover, the ground state $Q$ attains the best constant in the Gagliardo-Nirenberg inequality
\[
\left\|v\right\|_{2+\frac{4}{N}}^{2+\frac{4}{N}}\leq\left(1+\frac{2}{N}\right)\left(\frac{\left\|v\right\|_2}{\left\|Q\right\|_2}\right)^{\frac{4}{N}}\left\|\nabla v\right\|_2^2\quad\text{for }v\in H^1(\mathbb{R}^N).
\]
Therefore, for all $v\in H^1(\mathbb{R}^N)$,
\[
E_{\text{crit}}(v)\geq \frac{1}{2}\left\|\nabla v\right\|_2^2\left(1-\left(\frac{\left\|v\right\|_2}{\left\|Q\right\|_2}\right)^{\frac{4}{N}}\right)
\]
holds. This inequality and the mass and energy conservations imply that any subcritical-mass solution for \eqref{CNLS} is global and bounded in $H^1(\mathbb{R}^N)$.

Regarding the critical mass case, we apply the pseudo-conformal transformation
\[
u(t,x)\ \mapsto\ \frac{1}{\left|t\right|^\frac{N}{2}}u\left(-\frac{1}{t},\pm\frac{x}{t}\right)e^{i\frac{\left|x\right|^2}{4t}}
\]
to the solitary wave solution $u(t,x):=Q(x)e^{it}$. Then we obtain
\[
S(t,x):=\frac{1}{\left|t\right|^\frac{N}{2}}Q\left(\frac{x}{t}\right)e^{-\frac{i}{t}}e^{i\frac{\left|x\right|^2}{4t}},
\]
which is also a solution for \eqref{CNLS} and satisfies
\[
\left\|S(t)\right\|_2=\left\|Q\right\|_2,\quad \left\|\nabla S(t)\right\|_2\sim\frac{1}{\left|t\right|}\quad (t\nearrow 0).
\]
Namely, $S$ is a minimal-mass blow-up solution for \eqref{CNLS}. Moreover, $S$ is the only finite time blow-up solution for \eqref{CNLS} with critical mass, up to the symmetries of the flow (see \cite{MMMB}).

Regarding the supercritical mass case, there exists a solution $u$ for \eqref{CNLS} such that
\[
\left\|\nabla u(t)\right\|_2\sim\sqrt{\frac{\log\bigl|\log\left|T^*-t\right|\bigr|}{T^*-t}}\quad (t\nearrow T^*)
\]
(see \cite{MRUPB,MRUDB}).

\subsection{Previous results}
We describe previous results \cite{LMR} regarding \eqref{NLS} with $C_2=0$:
\begin{align}
\label{DPNLS}
i\frac{\partial u}{\partial t}+\Delta u+|u|^{\frac{4}{N}}u\pm |u|^{p-1}u=0,\quad 1<p<1+\frac{4}{N}.
\end{align}

\begin{theorem}[\cite{LMR}, see also \cite{MIP}]
\label{Thm:DPexist}
For any energy level $E_0\in\mathbb{R}$, there exist $t_0<0$ and a radially symmetric initial value $u(t_0)\in \Sigma^1$ with
\[
\|u_0\|_2=\|Q\|_2,\quad E(u_0)=E_0
\]
such that the corresponding solution $u$ for \eqref{DPNLS} with $\pm=+$ blows up at $T^*=0$. Moreover,
\[
\left\|u(t)-\frac{1}{\lambda(t)^\frac{N}{2}}P\left(t,\frac{x}{\lambda(t)}\right)e^{-i\frac{b(t)}{4}\frac{|x|^2}{\lambda(t)^2}+i\gamma(t)}\right\|_{\Sigma^1}\rightarrow 0\quad (t\nearrow 0)
\]
holds for some blow-up profile $P$, positive constants $C_1(p)$ and $C_2(p)$, positive-valued $C^1$ function $\lambda$, and real-valued $C^1$ functions $b$ and $\gamma$ such that
\begin{align*}
P(t)&\rightarrow Q\text{ in }H^1(\mathbb{R}^N),&\lambda(t)&=C_1(p)|t|^{\frac{4}{4+N(p-1)}}\left(1+o(1)\right),\\
b(t)&=C_2(p)|t|^{\frac{4-N(p-1)}{4+N(p-1)}}\left(1+o(1)\right),&\gamma(t)^{-1}&=O\left(|t|^{\frac{4-N(p-1)}{4+N(p-1)}}\right)
\end{align*}
as $t\nearrow 0$.
\end{theorem}

\begin{theorem}[\cite{LMR}]
\label{Thm:DPnonexist}
For any critical-mass initial value $u(t_0)\in H^1(\mathbb{R}^N)$, the corresponding solution for \eqref{DPNLS} with $\pm=-$ is global and bounded in $H^1(\mathbb{R}^N)$.
\end{theorem}

Similarly to the critical problem, by using the Gagliardo-Nirenberg inequality, we can show that the subcritical-mass solution for \eqref{DPNLS} is global and bounded in $H^1(\mathbb{R}^N)$. Therefore, if there is a minimal-mass blow-up solution, it has a mass greater than or equal to critical mass. In Theorem \ref{Thm:DPexist}, a critical-mass blow-up solution with a blow-up rate of $|t|^{\frac{4}{4+N(p-1)}}$ has been constructed. This blow-up rate is different from the blow-up rate $t^{-1}$ of the critical problem. On the other hand, Theorem \ref{Thm:DPnonexist} shows that there is no blow-up solution with critical mass. For any supercritical-mass, there exists a blow-up solution for \eqref{DPNLS} with $\pm=-$ with that mass \cite[Lemma 1.2]{LMR}. Therefore, Theorem \ref{Thm:DPexist} states that there is no minimal-mass blow-up solution. Consequently, we see that the perturbation term $|u|^{p-1}u$ affects the existence of the minimal-mass blow solution and its behaviour.

Next, we describe previous result \cite{MIP} regarding \eqref{NLS} with $C_1=0$:
\begin{align}
\label{IPNLS}
i\frac{\partial u}{\partial t}+\Delta u+|u|^{\frac{4}{N}}u\pm \frac{1}{|x|^{2\sigma}}u=0.
\end{align}

\begin{theorem}[\cite{MIP}]
\label{Thm:IPexist}
Assume $0<\sigma<\min\left\{\frac{N}{4},1\right\}$. Then for any energy level $E_0\in\mathbb{R}$, there exist $t_0<0$ and a radially symmetric initial value $u_0\in \Sigma^1$ with
\[
\|u_0\|_2=\|Q\|_2,\quad E(u_0)=E_0
\]
such that the corresponding solution $u$ for \eqref{IPNLS} with $\pm=+$ and $u(t_0)=u_0$ blows up at $T^*=0$. Moreover,
\[
\left\|u(t)-\frac{1}{\lambda(t)^\frac{N}{2}}P\left(t,\frac{x}{\lambda(t)}\right)e^{-i\frac{b(t)}{4}\frac{|x|^2}{\lambda(t)^2}+i\gamma(t)}\right\|_{\Sigma^1}\rightarrow 0\quad (t\nearrow 0)
\]
holds for some blow-up profile $P$ and $C^1$ functions $\lambda:(t_0,0)\rightarrow(0,\infty)$ and $b,\gamma:(t_0,0)\rightarrow\mathbb{R}$ such that
\begin{align*}
P(t)&\rightarrow Q\quad\text{in}\ H^1(\mathbb{R}^N),&\lambda(t)&=C_1(\sigma)|t|^{\frac{1}{1+\sigma}}\left(1+o(1)\right),\\
b(t)&=C_2(\sigma)|t|^{\frac{1-\sigma}{1+\sigma}}\left(1+o(1)\right),& \gamma(t)^{-1}&=O\left(|t|^{\frac{1-\sigma}{1+\sigma}}\right)
\end{align*}
as $t\nearrow 0$.
\end{theorem}

\begin{theorem}[\cite{MIP}]
\label{Thm:IPnonexist}
Assume $N\geq 2$ and $0<\sigma<\min\left\{\frac{N}{2},1\right\}$. For any critical-mass initial value $u(t_0)\in H^1_{\text{rad}}(\mathbb{R}^N)$, the corresponding solution for \eqref{IPNLS} with $\pm=-$ is global and bounded in $H^1(\mathbb{R}^N)$.
\end{theorem}

For Theorems \ref{Thm:IPexist} and \ref{Thm:IPnonexist}, as in Theorems \ref{Thm:DPexist} and \ref{Thm:DPnonexist}, we see that the perturbation term $|x|^{-2\sigma}u$ affects the existence of the minimal-mass blow-up solution and its behaviour.

Let $\alpha_p$ and $\alpha_\sigma$ be defined by
\[
\alpha_p:=2-\frac{N}{2}(p-1),\quad \alpha_\sigma:=2-2\sigma.
\]
Then the blow-up rates of Theorem \ref{Thm:DPexist} and \ref{Thm:IPexist} are represented
\[
|t|^{-\frac{2}{4-\alpha_p}},\quad |t|^{-\frac{2}{4-\alpha_\sigma}},
\]
respectively. Therefore, if $\alpha_p=\alpha_\sigma$, then we expect a power nonlinearity $|u|^{p-1}u$ and an inverse power potential $|x|^{-2\sigma}$ to behave in a similar way. Assuming $\alpha_p=\alpha_\sigma$, the subcritical nonlinearity and the inverse power potential may influence each other to reach different conclusions from Theorems \ref{Thm:DPexist} and \ref{Thm:IPexist}. Therefore, in the following, we consider
\begin{align}
\label{NLS2}
i\frac{\partial u}{\partial t}+\Delta u+|u|^{\frac{4}{N}}u\pm C_0|u|^{p-1}u\mp\frac{1}{|x|^{2\sigma}}u=0\quad(\text{double-sign corresponds}),
\end{align}
where $C_0>0$ and $\alpha_p=\alpha_\sigma$.

Let $\alpha$ and $\omega$ be defined by
\[
\alpha:=\alpha_p=\alpha_\sigma,\quad \omega:=\frac{p+1}{2}\frac{\||\cdot|^{-\sigma}Q\|_2^2}{\|Q\|_{p+1}^{p+1}}.
\]

\subsection{Main results}
Firstly, we show that for \eqref{NLS2}, a minimal-mass solution, which blows up at a finite time is constructed when the attractive term is not inferior to the repulsive term (Theorems \ref{Thm:exist-1}, \ref{Thm:exist-2}, and \ref{Thm:exist-3}).

\begin{theorem}[Existence of a minimal-mass blow-up solution 1]
\label{Thm:exist-1}
Assume \eqref{index2}, $C_0>\omega$, and $\alpha_p=\alpha_\sigma$. Then for any energy level $E_0\in\mathbb{R}$, there exist $t_0<0$ and a radially symmetric initial value $u_0\in H^1(\mathbb{R}^N)$ with
\[
\|u_0\|_2=\|Q\|_2,\quad E(u_0)=E_0
\]
such that the corresponding solution $u$ for \eqref{NLS2} with $(\pm,\mp)=(+,-)$ and $u(t_0)=u_0$ blows up at $T^*=0$. Moreover,
\[
\left\|u(t)-\frac{1}{\lambda(t)^\frac{N}{2}}P\left(t,\frac{x}{\lambda(t)}\right)e^{-i\frac{b(t)}{4}\frac{|x|^2}{\lambda(t)^2}+i\gamma(t)}\right\|_{\Sigma^1}\rightarrow 0\quad (t\nearrow 0)
\]
holds for some blow-up profile $P$ and $C^1$ functions $\lambda:(t_0,0)\rightarrow(0,\infty)$ and $b,\gamma:(t_0,0)\rightarrow\mathbb{R}$ such that
\begin{align*}
P(t)&\rightarrow Q\quad\text{in}\ H^1(\mathbb{R}^N),&\lambda(t)&=C_1(\alpha)|t|^{\frac{2}{4-\alpha}}\left(1+o(1)\right),\\
b(t)&=C_2(\alpha)|t|^{\frac{\alpha}{4-\alpha}}\left(1+o(1)\right),& \gamma(t)^{-1}&=O\left(|t|^{\frac{\alpha}{4-\alpha}}\right)
\end{align*}
as $t\nearrow 0$.
\end{theorem}

\begin{theorem}[Existence of a minimal-mass blow-up solution 2]
\label{Thm:exist-2}
Assume \eqref{index2}, $0<C_0<\omega$, and $\alpha_p=\alpha_\sigma$. Then for any energy level $E_0\in\mathbb{R}$, there exist $t_0<0$ and a radially symmetric initial value $u_0\in H^1(\mathbb{R}^N)$ with
\[
\|u_0\|_2=\|Q\|_2,\quad E(u_0)=E_0
\]
such that the corresponding solution $u$ for \eqref{NLS2} with $(\pm,\mp)=(-,+)$ and $u(t_0)=u_0$ blows up at $T^*=0$. Moreover,
\[
\left\|u(t)-\frac{1}{\lambda(t)^\frac{N}{2}}P\left(t,\frac{x}{\lambda(t)}\right)e^{-i\frac{b(t)}{4}\frac{|x|^2}{\lambda(t)^2}+i\gamma(t)}\right\|_{\Sigma^1}\rightarrow 0\quad (t\nearrow 0)
\]
holds for some blow-up profile $P$ and $C^1$ functions $\lambda:(t_0,0)\rightarrow(0,\infty)$ and $b,\gamma:(t_0,0)\rightarrow\mathbb{R}$ such that
\begin{align*}
P(t)&\rightarrow Q\quad\text{in}\ H^1(\mathbb{R}^N),&\lambda(t)&=C_1(\alpha)|t|^{\frac{2}{4-\alpha}}\left(1+o(1)\right),\\
b(t)&=C_2(\alpha)|t|^{\frac{\alpha}{4-\alpha}}\left(1+o(1)\right),& \gamma(t)^{-1}&=O\left(|t|^{\frac{\alpha}{4-\alpha}}\right)
\end{align*}
as $t\nearrow 0$.
\end{theorem}

In particular, when the attractive term is balanced by the repulsive term, there is a minimal-mass blow-up solution with a blow-up rate $t^{-1}$ like in the critical problem:

\begin{theorem}[Existence of a minimal-mass blow-up solution 3]
\label{Thm:exist-3}
Assume \eqref{index2}, $C_0=\omega$, and $\alpha:=\alpha_p=\alpha_\sigma>1$. Then for any energy level $E_0>0$, there exist $t_0<0$ and a radially symmetric initial value $u_0\in H^1(\mathbb{R}^N)$ with
\[
\|u_0\|_2=\|Q\|_2,\quad E(u_0)=E_0
\]
such that the corresponding solution $u$ for \eqref{NLS2} with $u(t_0)=u_0$ blows up at $T^*=0$. Moreover,
\[
\left\|u(t)-\frac{1}{\lambda(t)^\frac{N}{2}}Q\left(t,\frac{x}{\lambda(t)}\right)e^{-i\frac{b(t)}{4}\frac{|x|^2}{\lambda(t)^2}+i\gamma(t)}\right\|_{\Sigma^1}\rightarrow 0\quad (t\nearrow 0)
\]
holds for $C^1$ functions $\lambda:(t_0,0)\rightarrow(0,\infty)$ and $b,\gamma:(t_0,0)\rightarrow\mathbb{R}$ such that
\begin{align*}
\lambda(t)=\sqrt{\frac{8E_0}{\|yQ\|_2^2}}|t|\left(1+o(1)\right),\quad b(t)=\frac{8E_0}{\|yQ\|_2^2}|t|\left(1+o(1)\right),\quad \gamma(t)^{-1}=O\left(|t|\right)
\end{align*}
as $t\nearrow 0$.
\end{theorem}

On the other hand, minimal-mass blow-up solutions do not exist when the attractive term is inferior to the repulsive term (Theorems \ref{Thm:nonexist-1} and \ref{Thm:nonexist-2}).

\begin{theorem}[Nonexistence of a minimal-mass blow-up solution 1]
\label{Thm:nonexist-1}
Assume \eqref{index1}, $C_0>\omega$, and $\alpha_p=\alpha_\sigma$. Then for any critical-mass initial value $u(t_0)\in H^1(\mathbb{R}^N)$, the corresponding solution for \eqref{NLS2} with $(\pm,\mp)=(-,+)$ is global and bounded in $H^1(\mathbb{R}^N)$.
\end{theorem}

\begin{theorem}[Nonexistence of a minimal-mass blow-up solution 2]
\label{Thm:nonexist-2}
Assume $N\geq 2$, \eqref{index2}, $0<C_0<\omega$, and $\alpha_p=\alpha_\sigma$. Then for any critical-mass initial value $u(t_0)\in H^1_{\text{rad}}(\mathbb{R}^N)$, the corresponding solution for \eqref{NLS2} with $(\pm,\mp)=(+,-)$ is global and bounded in $H^1(\mathbb{R}^N)$.
\end{theorem}

Furthermore, optimal lower estimates are given for blow-up rates of critical-mass solution that blows up at a finite time in general (Theorems \ref{Thm:propmmbs-1} and \ref{Thm:propmmbs-2}). In particular, optimal lower estimations of blow-up rates are almost unknown except for classical results and cases where the uniqueness of the blow-up solution is known. 

\begin{theorem}[Behaviour of a minimal-mass blow-up solution 1]
\label{Thm:propmmbs-1}
Assume $N\geq 2$ and that $u$ is a radial $H^1$-solution for \eqref{NLS2} with $C_0\in(0,\omega)\cup(\omega,\infty)$, \eqref{DPNLS} with $\pm=+$, or \eqref{IPNLS} with $\pm=+$ which has critical mass and blows up at a finite time $T$. Then
\[
\|\nabla u(t)\|_2\gtrsim\frac{1}{|T-t|^\frac{2}{4-\alpha}}\quad(t\rightarrow T)
\]
holds.
\end{theorem}

\begin{theorem}[Behaviour of a minimal-mass blow-up solution 2]
\label{Thm:propmmbs-2}
Assume $N\geq 2$ and that $u$ is a radial $H^1$-solution for \eqref{NLS2} with $C_0=\omega$ and $\|u(0)\|_2=\|Q\|_2$. Then
\[
E(u)>0
\] 
holds. Moreover, if $u$ blows up at a finite time $T$, then the blow-up rate is estimated by
\[
\|\nabla u(t)\|_2\gtrsim\frac{1}{|T-t|}\quad(t\rightarrow T).
\]
\end{theorem}

Theorems \ref{Thm:exist-1}, \ref{Thm:exist-2}, \ref{Thm:nonexist-1}, and \ref{Thm:nonexist-2} are shown in the same way as in the proofs in \cite{LMR,MIP}. Therefore, their proofs are only leave in Remark \ref{betapositive}. Henceforth, we only prove Theorems \ref{Thm:exist-3}, \ref{Thm:propmmbs-2}, and \ref{Thm:propmmbs-1}.

\subsection{Comments regarding the main results}
Theorem \ref{Thm:exist-3} implies that the power nonlinearity and the inverse power potential behave in the same way with respect to blow-up. In this paper, we construct blow-up profile $P$ as an approximate solution for
\[
i\frac{\partial P}{\partial s}+\Delta P-P+|P|^\frac{4}{N}P\pm C_0\lambda^{\alpha_p}|P|^{p-1}P\mp\lambda^{\alpha_\sigma} \frac{1}{|y|^{2\sigma}}P=0
\]
and treat $\lambda^{\alpha_p}|P|^{p-1}P$ and $\lambda^{\alpha_\sigma} \frac{1}{|y|^{2\sigma}}P$ as perturbation terms. If $\alpha_p\neq\alpha_\sigma$, then we expect that the blow-up is dominated by the perturbation term for which $\alpha$ is smaller since $\lambda(s)\rightarrow0$ as $s\rightarrow\infty$. Namely, it is expected that similar conclusion to the result in \cite{LMR} or \cite{MIP} will hold. Likewise, if $\alpha_p=\alpha_\sigma$, then we expect that the blow-up for \eqref{NLS} is dominated by the perturbation term for which $C_j$ is larger. Theorems \ref{Thm:exist-1}, \ref{Thm:exist-2}, \ref{Thm:nonexist-1}, and \ref{Thm:nonexist-2} imply that the threshold of these results is $C_1=\omega C_2$. Moreover, Theorem \ref{Thm:exist-3} states that \eqref{NLS} with $C_1\pm \omega$ and $C_2=\mp1$ (i.e., \eqref{NLS2} with $C_0=\omega$) has a minimal-mass blow-up solution with the same blow-up rate $t^{-1}$ as for the mass-critical problem \eqref{CNLS}.

In Theorem \ref{Thm:exist-3}, unlike Theorems \ref{Thm:exist-1} and \ref{Thm:exist-2}, there is a condition that the energy level $E_0$ is positive. In \cite{LMR,MIP}, $\lambda_1$ and $b_1$ are defined by
\[
s_1:=\int_{\lambda_1}^{\lambda_0}\frac{1}{\mu^{\frac{\alpha}{2}+1}\sqrt{\frac{2\beta}{2-\alpha}+\frac{8E_0}{\|yQ\|_2^2}\mu^{2-\alpha}}}d\mu,\quad E(P_{\lambda_1,b_1,0})=E_0
\]
for sufficiently large $s_1$. Since $\beta=0$ holds under the assumptions of Theorem \ref{Thm:exist-3}, a formal application of this definition requires $E_0>0$. Thus, in the proof of Theorem \ref{Thm:exist-3}, the condition is merely a technical requirement. However, Theorem \ref{Thm:propmmbs-2} states that the critical-mass blow-up solution for \eqref{NLS2} (not necessarily a finite time blow-up) has always positive energy. Therefore, the assumption of positive energy in Theorem \ref{Thm:exist-3} is inevitable.

The lower estimation in Theorem \ref{Thm:propmmbs-2} or \ref{Thm:propmmbs-1} could be shown by modifying Lemma \ref{decomposition} and separating Lemma \ref{Modesti-1} from Lemma \ref{Modesti-2}. Firstly, assuming that the solution blows up at $T$, it can be decomposed by using Lemma \ref{decomposition}. The parameters $\tilde{\lambda}$ and $\tilde{\varepsilon}$ of the decomposition are expected to converge to $0$ as $t\to T$. We have corrected the ``There exist $\overline{C}>0$'' in the old decomposition lemma (e.g., \cite[Lemma 4.1]{MIP}) to ``For any $\epsilon_0$'' in Lemma \ref{decomposition}, thus easily showing that the parameters $\tilde{\lambda}$ and $\tilde{\varepsilon}$ converge to $0$ as $t\to T$. If we want to show the uniqueness of blow-up rates, we need to estimate this parameters $\tilde{\lambda}$ and $\tilde{\varepsilon}$. For this purpose, one can consider using a bootstrap, as in the case of the construction of a minimal-mass blow-up solution. However, while this bootstrap assumes $\tilde{\varepsilon}(t_1)=0$, it is not clear whether this holds for $\tilde{\varepsilon}$ of general critical-mass blow-up solutions. Therefore, we avoid this approach and partially solve the problem by using Lemma \ref{Modesti-1}. In particular, the lower estimations in Theorems \ref{Thm:propmmbs-2} and \ref{Thm:propmmbs-1} are optimal estimations, because the solutions with the blow-up rates are actually constructed in Theorems \ref{Thm:exist-1}, \ref{Thm:exist-2}, and \ref{Thm:exist-3}.

Let $u$ be blow-up solution for
\[
i\frac{\partial u}{\partial t}+\Delta u+\mu|u|^{p-1}u=0,\quad 1+\frac{4}{N}\leq p<\frac{4}{N-2}\quad\left(1+\frac{4}{N}\leq p<\infty\ \text{if}\ N=1\right),
\]
where $\mu>0$. It is known that
\[
\|\nabla u(t)\|_2\gtrsim\frac{1}{|T-t|^{\frac{1}{p-1}-\frac{N-2}{4}}}\quad\text{as}\ t\to T
\]
if $u$ blows up at a finite time $T$ (see \cite[Theorem 6.5.13]{CSSE}). In particular, when $p=\frac{4}{N}+1$, the lower estimation is as follows:
\[
\|\nabla u(t)\|_2\gtrsim\frac{1}{|T-t|^{\frac{1}{2}}}\quad\text{as}\ t\to T.
\]
Compared to this lower estimate, the lower estimate in Theorem \ref{Thm:propmmbs-2} is better, although there is a requirement that the solution has a critical mass.

In Theorem \ref{Thm:exist-3}, $\alpha>1$ is required. Assuming $\alpha<1$, there may be a minimal-mass blow-up solution with a blow-up rate that is not $t^{-1}$. For $P_{0,0}^+$, which constitutes the blow-up profile $P$ (see \eqref{Pdef}), if $\left\langle L_+P_{0,0}^+,P_{0,0}^+\right\rangle\neq0$, then we obtain
\[
\left\langle L_+P_{0,0}^+,P_{0,0}^+\right\rangle>0
\]
since $(P_{0,0}^+,Q)_2=0$. Therefore, $\beta_{0,1}>0$ and
\[
\lambda_{\text{app}}(s)=\left(\alpha\sqrt{\frac{\beta_{0,1}}{1-\alpha}}\right)^{-\frac{1}{\alpha}}s^{-\frac{2}{\alpha}},\quad b_{\text{app}}(s)=\frac{1}{\alpha s}
\]
are solutions for
\[
\frac{\partial b}{\partial s}+b^2-\beta_{0,1}\lambda^{2\alpha}=0,\quad \frac{1}{\lambda}\frac{\partial \lambda}{\partial s}+b=0
\]
in $s>0$. Accordingly, we expect the existence of a minimal-mass blow-up solution with a blow-up rate $|t|^{-\frac{1}{2-\alpha}}$.

\subsection{Notations}
In this section, we introduce the notation used in this paper.

Let
\[
\mathbb{N}:=\mathbb{Z}_{\geq 1},\quad\mathbb{N}_0:=\mathbb{Z}_{\geq 0}.
\]
We define
\begin{align*}
(u,v)_2&:=\re\int_{\mathbb{R}^N}u(x)\overline{v}(x)dx,&\left\|u\right\|_p&:=\left(\int_{\mathbb{R}^N}|u(x)|^pdx\right)^\frac{1}{p},&&&&\\
f(z)&:=|z|^\frac{4}{N}z,&F(z)&:=\frac{1}{2+\frac{4}{N}}|z|^{2+\frac{4}{N}}&g(z)&:=|z|^{p-1}z,&G(z)&:=\frac{1}{p+1}|z|^{p+1}\quad \text{for $z\in\mathbb{C}$}.
\end{align*}
By identifying $\mathbb{C}$ with $\mathbb{R}^2$, we denote the differentials of $f$, $g$, $F$, and $G$ by $df$, $dg$, $dF$, and $dG$ , respectively. We define
\[
\Lambda:=\frac{N}{2}+x\cdot\nabla,\quad L_+:=-\Delta+1-\left(1+\frac{4}{N}\right)Q^\frac{4}{N},\quad L_-:=-\Delta+1-Q^\frac{4}{N}.
\]
Namely, $\Lambda$ is the generator of $L^2$-scaling, and $L_+$ and $L_-$ come from the linearised Schr\"{o}dinger operator to around $Q$. Then
\[
L_-Q=0,\quad L_+\Lambda Q=-2Q,\quad L_-|x|^2Q=-4\Lambda Q,\quad L_+\rho=|x|^2 Q
\]
hold, where $\rho\in\mathcal{S}(\mathbb{R}^N)$ is the unique radial solution for $L_+\rho=|x|^2 Q$. Note that there exist $C_\alpha,\kappa_\alpha>0$ such that
\[
\left|\left(\frac{\partial}{\partial x}\right)^\alpha Q(x)\right|\leq C_\alpha Q(x),\quad \left|\left(\frac{\partial}{\partial x}\right)^\alpha \rho(x)\right|\leq C_\alpha(1+|x|)^{\kappa_\alpha} Q(x).
\]
for any multi-index $\alpha$. Furthermore, there exists $\mu>0$ such that for all $u\in H_{\text{rad}}^1(\mathbb{R}^N)$,
\begin{align}
\label{Lcoer}
&\left\langle L_+\re u,\re u\right\rangle+\left\langle L_-\im u,\im u\right\rangle\nonumber\\
\geq&\ \mu\left\|u\right\|_{H^1}^2-\frac{1}{\mu}\left({(\re u,Q)_2}^2+{(\re u,|x|^2 Q)_2}^2+{(\im u,\rho)_2}^2\right)
\end{align}
(e.g., see \cite{MRO,MRUPB,RSEU,WL}). We denote by $\mathcal{Y}$ the set of functions $g\in C^{\infty}(\mathbb{R}^N\setminus\{0\})\cap C(\mathbb{R}^N)\cap H^1_{\text{rad}}(\mathbb{R}^N)$ such that
\[
\exists C_\alpha,\kappa_\alpha>0,\ |x|\geq 1\Rightarrow \left|\left(\frac{\partial}{\partial x}\right)^\alpha g(x)\right|\leq C_\alpha(1+|x|)^{\kappa_{\alpha}}Q(x)
\]
for any multi-index $\alpha$. Moreover, we defined by $\mathcal{Y}'$ the set of functions $g\in\mathcal{Y}$ such that
\[
\Lambda g\in H^1(\mathbb{R}^N)\cap C(\mathbb{R}^N).
\]

Finally, we use the notation $\lesssim$ and $\gtrsim$ when the inequalities hold up to a positive constant. We also use the notation $\approx$ when $\lesssim$ and $\gtrsim$ hold. Moreover, positive constants $C$ and $\epsilon$ are sufficiently large and small, respectively.

\section{Construction of a blow-up profile}
For $K\in\mathbb{N}_0$, let
\[
\Sigma_K=\{\ (j,k)\in{\mathbb{N}_0}^2\ |\ j+k\leq K\ \}.
\]

\begin{proposition}[Existence of a blow-up profile]
Let $K,K'\in\mathbb{N}_0$ be sufficiently large. Let $\lambda(s)>0$ and $b(s)\in\mathbb{R}$ be $C^1$ functions of $s$ such that $\lambda(s)+|b(s)|\ll 1$. Then for any $(j,k)\in\Sigma_{K+K'}$, there exist $P_{j,k}^+,P_{j,k}^-\in\mathcal{Y}'$, $\beta_{j,k}\in\mathbb{R}$, and $\Psi:(\lambda,b)\mapsto \Psi(\cdot;\lambda,b)\in H^1(\mathbb{R}^N)$ such that $P$ satisfies
\begin{align}
\label{Peq}
i\frac{\partial P}{\partial s}+\Delta P-P+f(P)\pm C_0\lambda^\alpha g(P)\mp\lambda^\alpha \frac{1}{|y|^{2\sigma}}P+\theta\frac{|y|^2}{4}P=\Psi,
\end{align}
where $P$ and $\theta$ defined by
\begin{align}
\label{Pdef}
P(s,y)&=Q(y)+\sum_{(j,k)\in\Sigma_{K+K'}}\left(b(s)^{2j}\lambda(s)^{(k+1)\alpha}P_{j,k}^+(y)+ib(s)^{2j+1}\lambda(s)^{(k+1)\alpha}P_{j,k}^-(y)\right)\\
\theta(s)&=\sum_{(j,k)\in\Sigma_{K+K'}}b(s)^{2j}\lambda(s)^{(k+1)\alpha}\beta_{j,k}\nonumber
\end{align}
In particular,
\[
\beta_{0,0}=0
\]
holds.

Moreover, for some sufficiently small $\epsilon'>0$,
\begin{align}
\label{Psiesti}
\left\|e^{\epsilon'|\cdot|}\Psi\right\|_{H^1}\lesssim\lambda^\alpha\left(\left|b+\frac{1}{\lambda}\frac{\partial \lambda}{\partial s}\right|+\left|\frac{\partial b}{\partial s}+b^2-\theta\right|\right)+(b^2+\lambda^\alpha)^{K+2}
\end{align}
holds.
\end{proposition}

\begin{proof}
The proof is the same as for \cite{LMR,MIP}. We prove only $\beta_{0,0}=0$. From the proofs in \cite{LMR,MIP}, $P_{0,0}^+$ satisfies
\[
L_+P_{0,0}^+-\beta_{0,0}\frac{|y|^2}{4}Q\mp C_0g(Q)\pm\frac{1}{|y|^{2\sigma}}Q=0,\quad (P_{0,0}^+,Q)_2=0.
\]
Therefore, since $C_0=\omega$,
\begin{align*}
0&=\left(P_{0,0}^+,Q\right)_2=-\frac{1}{2}\left\langle L_+P_{0,0}^+,\Lambda Q\right\rangle=-\frac{1}{2}\left\langle\beta_{0,0}\frac{|y|^2}{4}Q\pm C_0g(Q)\mp\frac{1}{|y|^{2\sigma}}Q,\Lambda Q\right\rangle\\
&=\frac{1}{2}\left(\frac{\beta_{0,0}}{4}\||\cdot|Q\|_2^2\mp C_0\frac{N(p-1)}{2(p+1)}\|Q\|_{p+1}^{p+1}\pm \sigma\||\cdot|^{-\sigma}Q\|_2^2\right).
\end{align*}
Accordingly, $\beta_{0,0}=0$.
\end{proof}

\begin{remark}
\label{betapositive}
Assume $(\pm,\mp)=(+,-)$ and $C_0>\omega$. Then we obtain
\[
0=\left(P_{0,0}^+,Q\right)_2=\frac{1}{2}\left(\frac{\beta_{0,0}}{4}\||\cdot|Q\|_2^2-(C_0-\omega)\frac{N(p-1)}{2(p+1)}\|Q\|_{p+1}^{p+1}\right).
\]
Therefore, $\beta_{0,0}>0$. Likewise, if $(\pm,\mp)=(-,+)$ and $C_0<\omega$, we obtain $\beta_{0,0}>0$. Thus, Theorems \ref{Thm:exist-1} and \ref{Thm:exist-2} can be shown as in \cite{LMR,MIP}.
\end{remark}

For the blow-up profile $P$, the following properties are obtained by direct calculation:

\begin{proposition}
\label{Pprop}
Let define
\[
P_{\lambda,b,\gamma}(s,x):=\frac{1}{\lambda(s)^\frac{N}{2}}P\left(s,\frac{x}{\lambda(s)}\right)e^{-i\frac{b(s)}{4}\frac{|x|^2}{\lambda(s)^2}+i\gamma(s)}.
\]
Then,
\begin{align*}
\left|\frac{d}{ds}\|P_{\lambda,b,\gamma}\|_2^2\right|&\lesssim\lambda^\alpha\left(\left|b+\frac{1}{\lambda}\frac{\partial \lambda}{\partial s}\right|+\left|\frac{\partial b}{\partial s}+b^2-\theta\right|\right)+(b^2+\lambda^\alpha)^{K+2},\\
\left|\frac{d}{ds}E(P_{\lambda,b,\gamma})\right|&\lesssim\frac{1}{\lambda^2}\left(\left|b+\frac{1}{\lambda}\frac{\partial \lambda}{\partial s}\right|+\left|\frac{\partial b}{\partial s}+b^2-\theta\right|+(b^2+\lambda^\alpha)^{K+2}\right)
\end{align*}
hold. Moreover,
\begin{eqnarray}
\label{Eesti}
\left|8E(P_{\lambda,b,\gamma})-\||\cdot|Q\|_2^2\frac{b^2}{\lambda^2}\right|\lesssim\frac{\lambda^\alpha(b^2+\lambda^\alpha)}{\lambda^2}
\end{eqnarray}
holds.
\end{proposition}

This section closes with a decomposition lemma, which is key to this paper. See \cite{MRUPB,MIP} for the proof.

\begin{lemma}[Decomposition]
\label{decomposition}
For any $\epsilon_0>0$, there exist $\overline{l},\delta>0$ such that the following statement.
Let $I$ be an interval. We assume that $u\in C(I,H^1(\mathbb{R}^N))\cap C^1(I,H^{-1}(\mathbb{R}^N))$ satisfies
\[
\forall\ t\in I,\ \left\|\lambda(t)^{\frac{N}{2}}u\left(t,\lambda(t)y\right)e^{i\gamma(t)}-Q\right\|_{H^1}< \delta
\]
for some functions $\lambda:I\rightarrow(0,\overline{l})$ and $\gamma:I\rightarrow\mathbb{R}$. Then there exist unique functions $\tilde{\lambda}:I\rightarrow(0,\infty)$, $\tilde{b}:I\rightarrow\mathbb{R}$, and $\tilde{\gamma}:I\rightarrow\mathbb{R}\slash 2\pi\mathbb{Z}$ such that 
\begin{align}
\label{mod}
&u(t,x)=\frac{1}{\tilde{\lambda}(t)^{\frac{N}{2}}}\left(P+\tilde{\varepsilon}\right)\left(t,\frac{x}{\tilde{\lambda}(t)}\right)e^{-i\frac{\tilde{b}(t)}{4}\frac{|x|^2}{\tilde{\lambda}(t)^2}+i\tilde{\gamma}(t)},\\
&\left|\frac{\tilde{\lambda}(t)}{\lambda(t)}-1\right|+\left|\tilde{b}(t)\right|+\left|\tilde{\gamma}(t)+\gamma(t)\right|_{\mathbb{R}\slash 2\pi\mathbb{Z}}<\epsilon_0\nonumber
\end{align}
hold, where $|\cdot|_{\mathbb{R}\slash 2\pi\mathbb{Z}}$ is defined by
\[
|c|_{\mathbb{R}\slash 2\pi\mathbb{Z}}:=\inf_{m\in\mathbb{Z}}|c+2\pi m|,
\]
and that $\tilde{\varepsilon}$ satisfies the orthogonal conditions
\begin{align}
\label{orthocondi}
\left(\tilde{\varepsilon},i\Lambda P\right)_2=\left(\tilde{\varepsilon},|y|^2P\right)_2=\left(\tilde{\varepsilon},i\rho\right)_2=0
\end{align}
on $I$. In particular, $\tilde{\lambda}$, $\tilde{b}$, and $\tilde{\gamma}$ are $C^1$ functions and independent of $\lambda$ and $\gamma$.
\end{lemma}

\begin{remark*}
This lemma is slightly different from \cite[Lemma 4.1]{MIP}, with ``There exist $\overline{C}>0$'' changed to ``For any $\epsilon_0$''. Not essential for the construction of a minimal-mass blow-up solution, but required in the proof of Theorems \ref{Thm:propmmbs-1} and \ref{Thm:propmmbs-2} (see \eqref{paramconv}).

This modification follows easily from the implicit function theorem. Indeed, for a $C^1$-function $f:U\to\mathbb{R}^{n+m}$ with $f(x_0,y_0)=0$ and $\det\frac{\partial f}{\partial y}(x_0,y_0)\neq 0$, from the implicit function theorem, there exist a neighbourhood $V$ of $x_0$, a neighbourhood $W$ of $y_0$, and a $C^1$-function $S:V\to W$ such that $f(x,S(x))=0$ for any $x\in V$. In particular, from the continuity and uniqueness of $S$, for any $\varepsilon>0$, there exists some $\delta>0$ and $S:B(x_0,\delta)\to B(y_0,\varepsilon)$ also $f(x,S(x))=0$ for any $x\in B(x_0,\delta)$. This $\varepsilon$ corresponds to $\epsilon_0$ in Lemma \ref{decomposition} and $\delta$ corresponds to $\delta$ and $\overline{l}$ in Lemma \ref{decomposition}.
\end{remark*}

\section{Uniformity estimates for decomposition parameters}
\label{sec:uniesti}
For $s_1>0$, let $\lambda_1$ and $b_1>0$ be defined by
\begin{align*}
\label{lambini}
\lambda_1:=\sqrt{\frac{\|yQ\|_2^2}{8E_0}}{s_1}^{-1},\quad E(P_{\lambda_1,b_1,0})=E_0.
\end{align*}
Note that such $b_1>0$ exists from \eqref{Eesti}. Moreover, let functions $\lambda_{\app}$ and $b_{\app}$ be defined by
\[
\lambda_{\app}(s):=\sqrt{\frac{\|yQ\|_2^2}{8E_0}}s^{-1},\quad b_{\app}(s):=s^{-1}.
\]

Let $u(t)$ be the solution for \eqref{NLS2} with an initial value
\begin{align}
u(t_1,x):=\frac{1}{{\lambda_1}^\frac{N}{2}}P\left(\frac{x}{\lambda_1}\right)e^{-i\frac{b_1}{4}\frac{|x|^2}{{\lambda_1}^2}}.
\end{align}
Then since $u$ satisfies the assumption of Lemma \ref{decomposition} in a neighbourhood of $t_1$, there exists a decomposition $(\tilde{\lambda}_{t_1},\tilde{b}_{t_1},\tilde{\gamma}_{t_1},\tilde{\varepsilon}_{t_1})$ such that $(\ref{mod})$ in a neighbourhood $I$ of $t_1$. The rescaled time $s_{t_1}$ is defined by
\[
s_{t_1}(t):=s_1-\int_t^{t_1}\frac{1}{\tilde{\lambda}_{t_1}(\tau)^2}d\tau.
\]
Moreover, let $I_{t_1}$ be the maximal interval such that a decomposition as $(\ref{mod})$ is obtained and we define 
\[
J_{s_1}:=s_{s_1}\left(I_{t_1}\right).
\]

Then, since $s_{t_1}:I_{t_1}\rightarrow J_{s_1}$ is strictly monotonically increasing, we can define inverse function ${s_{t_1}}^{-1}:J_{s_1}\rightarrow I_{t_1}$. Furthermore, we define
\begin{align*}
t_{t_1}&:=-\frac{\|yQ\|_2^2}{8E_0}{s_{t_1}}^{-1},& \lambda_{t_1}(s)&:=\tilde{\lambda}(t_{t_1}(s)),& b_{t_1}(s)&:=\tilde{b}(t_{t_1}(s)),\\
\gamma_{t_1}(s)&:=\tilde{\gamma}(t_{t_1}(s)),& \varepsilon_{t_1}(s,y)&:=\tilde{\varepsilon}(t_{t_1}(s),y).&&
\end{align*}
For the sake of clarity in notation, we often omit the subscript $t_1$. In particular, it should be noted that $u\in C((T_*,T^*),\Sigma^2(\mathbb{R}^N))$ and $|x|\nabla u\in C((T_*,T^*),L^2(\mathbb{R}^N))$. Additionally, let $s_0$ be sufficiently large, $s_1\geq s_0$, and
\[
s':=\max\left\{s_0,\inf J_{s_1}\right\}.
\]

Let $s_*$ be defined by
\begin{align}
\label{s*def}
s_*:=\inf\left\{\sigma\in(s',s_1]\ \middle|\ \text{(\ref{prebootstrap}) holds on }[\sigma,s_1]\right\},
\end{align}
where
\begin{align}
\label{prebootstrap}
\left\|\varepsilon(s)\right\|_{H^1}^2+b(s)^2\|y\varepsilon(s)\|_2^2<s^{-2K},\quad\left|\frac{\lambda(s)}{\lambda_{\text{app}}(s)}-1\right|+\left|\frac{b(s)}{b_{\text{app}}(s)}-1\right|<s^{-M}
\end{align}
with some
\[
0<M<2(\alpha-1).
\]

By direct calculation with \eqref{NLS2}, \eqref{mod}, and \eqref{Peq},
\begin{align}
\label{theorem:epsieq}
&i\frac{\partial \varepsilon}{\partial s}+\Delta \varepsilon-\varepsilon+f\left(P+\varepsilon\right)-f\left(P\right)\pm C_0\lambda^\alpha\left(g(P+\varepsilon)-g(P)\right)\mp\lambda^\alpha \frac{1}{|y|^{2\sigma}}\varepsilon+\theta\frac{|y|^2}{4}\varepsilon\\
&\hspace{20pt}-i\left(\frac{1}{\lambda}\frac{\partial \lambda}{\partial s}+b\right)\Lambda (P+\varepsilon)+\left(1-\frac{\partial \gamma}{\partial s}\right)(P+\varepsilon)\nonumber\\
&\hspace{40pt}+\left(\frac{\partial b}{\partial s}+b^2-\theta\right)\frac{|y|^2}{4}(P+\varepsilon)-\left(\frac{1}{\lambda}\frac{\partial \lambda}{\partial s}+b\right)b\frac{|y|^2}{2}(P+\varepsilon)\nonumber\\
&\hspace{60pt}=-\Psi\nonumber
\end{align}
holds in $J_{s_1}$.

\begin{lemma}
\label{Modesti-1}
For $s\in J_{s_1}$,
\[
|\Mod(s)|\lesssim \left|(\varepsilon,P)_2\right|+\lambda^\alpha\left\|\varepsilon e^{-ib\frac{|y|^2}{4}}\right\|_{H^1}+\left\|\varepsilon e^{-ib\frac{|y|^2}{4}}\right\|_{H^1}^2+(b^2+\lambda^\alpha)^{K+2}
\]
holds, where
\[
\Mod(s):=\left(\frac{1}{\lambda}\frac{\partial \lambda}{\partial s}+b,\frac{\partial b}{\partial s}+b^2,1-\frac{\partial \gamma}{\partial s}\right).
\]
\end{lemma}

\begin{proof}
See \cite{LMR,MIP} for details of the proof.

According to the orthogonality properties \eqref{orthocondi}, we have
\[
0=\frac{d}{ds}\left(i\varepsilon,\Lambda P\right)_2=\left(i\frac{\partial \varepsilon}{\partial s},\Lambda P\right)_2+\left(i\varepsilon,\frac{\partial (\Lambda P)}{\partial s}\right)_2.
\]

By direct calculation and \eqref{Pdef}, we obtain
\[
\left|\left(i\varepsilon,\frac{\partial (\Lambda P)}{\partial s}\right)_2\right|\lesssim \left|\Mod\right|\|\varepsilon\|_2+\lambda^\alpha\|\varepsilon\|_2.
\]

From \eqref{theorem:epsieq}, we obtain
\begin{align*}
&\left(i\frac{\partial \varepsilon}{\partial s},\Lambda P\right)_2\\
&=\left(L_+\re\varepsilon+iL_-\im\varepsilon-\left(f\left(P+\varepsilon\right)-f\left(P\right)-df(Q)(\varepsilon)\right)\mp C_0\lambda^\alpha\left(g(P+\varepsilon)-g(P)\right)\pm\lambda^\alpha \frac{1}{|y|^{2\sigma}}\varepsilon-\theta\frac{|y|^2}{4}\varepsilon\right.\\
&\hspace{40pt}\left.+i\left(\frac{1}{\lambda}\frac{\partial \lambda}{\partial s}+b\right)\Lambda (P+\varepsilon)-\left(1-\frac{\partial \gamma}{\partial s}\right)(P+\varepsilon)-\left(\frac{\partial b}{\partial s}+b^2-\theta\right)\frac{|y|^2}{4}(P+\varepsilon)\right.\\
&\hspace{80pt}\left.+\left(\frac{1}{\lambda}\frac{\partial \lambda}{\partial s}+b\right)b\frac{|y|^2}{2}(P+\varepsilon)+\Psi,\Lambda P\right)_2.
\end{align*}
Noting that
\[
\nabla \varepsilon=\nabla\left(\varepsilon e^{-ib\frac{|y|^2}{4}}\right)e^{ib\frac{|y|^2}{4}}+ib\frac{y}{2}\varepsilon,
\]
we obtain
\[
\left\langle L_+\re\varepsilon,\Lambda P\right\rangle=-2(\varepsilon,P)_2+O\left(\lambda^\alpha\left\|\varepsilon e^{-ib\frac{|y|^2}{4}}\right\|_{H^1}\right).
\]
Therefore, since $\left(|y|^2P,\Lambda P\right)_2=-\|yQ\|_2^2+O(\lambda^\alpha)$, we obtain
\begin{align*}
\left(i\frac{\partial \varepsilon}{\partial s},\Lambda P\right)_2=&-\frac{1}{4}\|yQ\|\left(\frac{\partial b}{\partial s}+b^2-\theta\right)-2(\varepsilon,P)_2+O\left(\lambda^\alpha\left\|\varepsilon e^{-ib\frac{|y|^2}{4}}\right\|_{H^1}\right)+O\left(\left\|\varepsilon e^{-ib\frac{|y|^2}{4}}\right\|_{H^1}^2\right)\\
&\hspace{20pt}+O\left((b^2+\lambda^\alpha)^{K+2}\right)+o\left(\left|\Mod\right|\right).
\end{align*}

Similar calculations are performed for other orthogonal conditions in \eqref{orthocondi} to obtain the estimate of Lemma \ref{Modesti-1}.
\end{proof}

\begin{remark}
\label{Modesti-1rem}
The estimate in Lemma \ref{Modesti-1} is independent of the initial values $\lambda_1$, $b_1$, and $0$ of $\lambda$, $b$, and $\varepsilon$, respectively. Therefore, a similar estimate can be obtained if $\lambda$, $b$, and $\varepsilon$ is sufficiently small.
\end{remark}

From Lemma \ref{Modesti-1}, we obtain the following lemma. See \cite{LMR,MIP} for details of the proof.

\begin{corollary}
\label{Modesti-2}
For $s\in(s_*,s_1]$,
\[
|(\varepsilon(s),Q)|\lesssim s^{-(K+\alpha)},\quad |\Mod(s)|\lesssim s^{-(K+\alpha)},\quad \|e^{\epsilon'|y|}\Psi\|_{H^1}\lesssim s^{-(K+2\alpha)}
\]
hold.
\end{corollary}

\begin{proof}
Let
\[
s_{**}:=\inf\left\{\ s\in[s_*,s_1]\ \middle|\ \left|(\varepsilon(\tau),P)_2\right|<\tau^{-(K+\alpha)}\ \mbox{holds on}\ [s,s_1].\ \right\}.
\]
Then from Lemma \ref{Modesti-1}, it is clearly
\[
|\Mod(s)|\lesssim s^{-(K+\alpha)}
\]
on $(s_{**},s_1]$. Therefore, from \eqref{Psiesti}, we obtain
\[
\|e^{\epsilon'|y|}\Psi\|_{H^1}\lesssim s^{-(K+2\alpha)}.
\]
Accordingly, from Proposition \ref{Pprop}, we obtain
\[
\left|(\varepsilon(\tau),P)_2\right|\lesssim s^{-(K+2\alpha-1)}.
\]
Namely, $s_*=s_{**}$ for sufficiently large $s_0$. Consequently, Corollary \ref{Modesti-2} holds.
\end{proof}

Let $m>0$ be sufficiently large and define
\begin{align*}
S(s,\varepsilon)&:=\frac{1}{\lambda^m}\biggl(\frac{1}{2}\left\|\varepsilon\right\|_{H^1}^2+b^2\left\|y\varepsilon\right\|_2^2-\int_{\mathbb{R}^N}\left(F(P+\varepsilon)-F(P)-dF(P)(\varepsilon)\right)dy\\
&\hspace{80pt}\mp C_0\lambda^\alpha\int_{\mathbb{R}^N}\left(G(P+\varepsilon)-G(P)-dG(P)(\varepsilon)\right)dy\pm\frac{1}{2}\lambda^\alpha\left\|y^{-\sigma}\varepsilon\right\|_2^2\biggr).
\end{align*}

\begin{lemma}[Estimates of $S$]
\label{Sesti}
For $s\in(s_*,s_1]$, 
\[
\frac{1}{\lambda^m}\left(\|\varepsilon\|_{H^1}^2+b^2\left\|y\varepsilon\right\|_2^2+O(s^{-2(K+\alpha)})\right)\lesssim S(s,\varepsilon)\lesssim \frac{1}{\lambda^m}\left(\|\varepsilon\|_{H^1}^2+b^2\left\|y\varepsilon\right\|_2^2\right)
\]
hold. Moreover,
\[
\frac{d}{ds}S(s,\varepsilon(s))\gtrsim \frac{b}{\lambda^m}\left(\|\varepsilon\|_{H^1}^2+b^2\left\|y\varepsilon\right\|_2^2+O(s^{-(2K+2\alpha-1)})\right)
\]
holds for $s\in(s_*,s_1]$.
\end{lemma}

\begin{proof}
The proof is the same as for \cite{LMR,MIP}.
\end{proof}

We use the estimates obtained in Lemma \ref{Sesti} and the bootstrap to establish the estimates of the parameters. Namely, we confirm \eqref{prebootstrap} on $[s_0,s_1]$.

\begin{lemma}[Re-estimation]
\label{rebootstrap}
For $s\in(s_*,s_1]$, 
\begin{align}
\label{reepsiesti}
&\left\|\varepsilon(s)\right\|_{H^1}^2+b(s)^2\left\|y\varepsilon(s)\right\|_2^2\lesssim s^{-(2K+\alpha)},\\
\label{reesti}
&\left|\frac{\lambda(s)}{\lambda_{\text{app}}(s)}-1\right|+\left|\frac{b(s)}{b_{\text{app}}(s)}-1\right|\lesssim s^{-2(\alpha-1)}.
\end{align}
\end{lemma}

\begin{proof}
See \cite{LMR,MIP} for details of the proof.

We prove \eqref{reepsiesti} by contradiction. Let $C_\dagger>0$ be sufficiently large and define
\[
s_\dagger:=\inf\left\{\sigma\in(s_*,s_1]\ \middle|\ \left\|\varepsilon(\tau)\right\|_{H^1}^2+b(\tau)^2\left\||y|\varepsilon(\tau)\right\|_2^2\leq C_\dagger \tau^{-(2K+\alpha)}\ (\tau\in[\sigma,s_1])\right\}.
\]
Then $s_\dagger<s_1$ holds. Here, we assume that $s_\dagger>s_*$ and define
\[
s_\ddagger:=\sup\left\{\sigma\in(s_*,s_1]\ \big|\ \left\|\varepsilon(\tau)\right\|_{H^1}^2+b(\tau)^2\left\||y|\varepsilon(\tau)\right\|_2^2\geq\tau^{-(2K+\alpha)}\ (\tau\in[s_\dagger,\sigma])\right\}.
\]
Then we obtain $s_\ddagger>s_\dagger$. Since $s\mapsto S(s,\varepsilon(s))$ is increasing on $[s_\dagger,s_\ddagger]$, we obtain
\[
C_1(C_\dagger-1)\leq 2C_2.
\]
It is a contradiction. Namely, $s_*=s_\dagger$.

From Proposition \ref{Pprop} and Lemma \ref{Modesti-2}, we obtain
\[
\left|E(P_{\lambda,b,\gamma})-E_0\right|\lesssim s^{-(K+\alpha-1)}\quad\text{and}\quad \left|b^2\|yQ\|_2^2-8\lambda^2E_0\right|\lesssim s^{-2\alpha}.
\]
Therefore,
\[
\left|\frac{\partial}{\partial s}\left(\sqrt{\frac{\|yQ\|_2^2}{8E_0}}\frac{1}{\lambda}-s\right)\right|\lesssim s^{-(2\alpha-1)},\quad\text{i.e., }\left|\frac{\lambda_{\mathrm{app}}(s)}{\lambda(s)}-1\right|\lesssim s^{-(2\alpha-1)}
\]
holds. Consequently, we obtain \eqref{reesti}.
\end{proof}

Furthermore, from Lemma \ref{rebootstrap} and \eqref{s*def} we obtain the following corollary:

\begin{corollary}
\label{s'=s_0}
If $s_0$ is sufficiently large, then $s_*=s'=s_0$ for any $s_1>s_0$.
\end{corollary}

Finally, we rewrite the estimates for $s$ in Lemma \ref{rebootstrap} into estimates for $t$.

\begin{lemma}
\label{interval}
Let $s_0$ be sufficiently large. Then there exists $t_0<0$ such that 
\[
[t_0,t_1]\subset {s_{t_1}}^{-1}([s_0,s_1]),\quad \left|\mathcal{C}s_{t_1}(t)^{-1}-|t|\right|\lesssim |t|^{M+1}\quad (t\in [t_0,t_1])
\]
hold for all $t_1\in(t_0,0)$, where $\mathcal{C}:=\frac{\|yQ\|_2^2}{8E_0}$.
\end{lemma}

Consequently, combining Lemma \ref{rebootstrap} and Lemma \ref{interval} the following lemma. See \cite{MP} for the proofs.

\begin{lemma}[Conversion of estimates]
\label{uniesti}
For any $t_1\in(t_0,0)$ and $t\in[t_0,t_1]$, 
\begin{align*}
\tilde{\lambda}_{t_1}(t)&=\sqrt{\frac{8E_0}{\|yQ\|_2^2}}|t|\left(1+\epsilon_{\tilde{\lambda},t_1}(t)\right),&\tilde{b}_{t_1}(t)&=\frac{8E_0}{\|yQ\|_2^2}|t|\left(1+\epsilon_{\tilde{b},t_1}(t)\right),\\
\|\tilde{\varepsilon}_{t_1}(t)\|_{H^1}&\lesssim |t|^{K+\frac{\alpha}{2}},&\|y\tilde{\varepsilon}_{t_1}(t)\|_{2}&\lesssim |t|^{K+\frac{\alpha}{2}-1}
\end{align*}
holds for some functions $\epsilon_{\tilde{\lambda},t_1}$ and $\epsilon_{\tilde{b},t_1}$. Furthermore,
\[
\sup_{t_1\in[t,0)}\left|\epsilon_{\tilde{\lambda},t_1}(t)\right|\lesssim |t|^M,\quad \sup_{t_1\in[t,0)}\left|\epsilon_{\tilde{b},t_1}(t)\right|\lesssim |t|^M.
\]
\end{lemma}

\section{Proof of Theorem \ref{Thm:exist-3}}
\label{sec:proof}
In this section, we prove Theorem \ref{Thm:exist-3}. See \cite{LMR,MP,MIP} for more details.

\begin{proof}[Proof of Theorem \ref{Thm:exist-3}]
Let $(t_n)_{n\in\mathbb{N}}\subset(t_0,0)$ be a increasing sequence such that $\lim_{n\nearrow \infty}t_n=0$. For each $n\in\mathbb{N}$, let $u_n$ be the solution for \eqref{NLS2} with the initial value
\begin{align*}
u_n(t_n,x):=\frac{1}{{\lambda_{1,n}}^\frac{N}{2}}P\left(\frac{x}{\lambda_{1,n}}\right)e^{-i\frac{b_{1,n}}{4}\frac{|x|^2}{{\lambda_{1,n}}^2}}
\end{align*}
at $t_n$, where
\[
s_n:=-\frac{\|yQ\|_2^2}{8E_0}{t_n}^{-1},\quad \lambda_n:=\sqrt{\frac{\|yQ\|_2^2}{8E_0}}{s_n}^{-1},\quad E(P_{\lambda_n,b_n,0})=E_0.
\]

According to Lemma \ref{decomposition}, there exists the decomposition
\[
u_n(t,x)=\frac{1}{\tilde{\lambda}_n(t)^{\frac{N}{2}}}\left(P+\tilde{\varepsilon}_n\right)\left(t,\frac{x}{\tilde{\lambda}_n(t)}\right)e^{-i\frac{\tilde{b}_n(t)}{4}\frac{|x|^2}{\tilde{\lambda}_n(t)^2}+i\tilde{\gamma}_n(t)}
\]
on $[t_0,t_n]$. Up to a subsequence, there exists $u_\infty(t_0)\in \Sigma^1$ such that
\[
u_n(t_0)\rightharpoonup u_\infty(t_0)\quad \text{weakly in}\ \Sigma^1,\quad u_n(t_0)\rightarrow u_\infty(t_0)\quad \text{in}\ L^2(\mathbb{R}^N)\quad (n\rightarrow\infty).
\]

Moreover, since $u_n:[t_0,0)\to\Sigma^1$ is locally uniformly bounded,
\[
u_n\rightarrow u_\infty\quad \text{in}\ C([t_0,T'],L^2(\mathbb{R}^N)),\quad u_n(t)\rightharpoonup u_\infty(t)\ \text{in}\ \Sigma^1 \quad (n\rightarrow\infty)
\]
holds (see \cite{MP}). Particularly, we have $\|u_\infty(t)\|_2=\|Q\|_2$.

According to weak convergence in $H^1(\mathbb{R}^N)$ and Lemma \ref{decomposition}, we decompose $u_\infty$ to
\[
u_\infty(t,x)=\frac{1}{\tilde{\lambda}_\infty(t)^{\frac{N}{2}}}\left(Q+\tilde{\varepsilon}_\infty\right)\left(t,\frac{x}{\tilde{\lambda}_\infty(t)}\right)e^{-i\frac{\tilde{b}_\infty(t)}{4}\frac{|x|^2}{\tilde{\lambda}_\infty(t)^2}+i\tilde{\gamma}_\infty(t)}
\]
on $[t_0,0)$. Furthermore, as $n\rightarrow\infty$, 
\begin{align*}
\tilde{\lambda}_n(t)\rightarrow\tilde{\lambda}_\infty(t),\quad \tilde{b}_n(t)\rightarrow \tilde{b}_\infty(t),\quad e^{i\tilde{\gamma}_n(t)}\rightarrow e^{i\tilde{\gamma}_\infty(t)},\quad\tilde{\varepsilon}_n(t)\rightharpoonup \tilde{\varepsilon}_\infty(t)\quad \text{weakly in}\ \Sigma^1
\end{align*}
hold for any $t\in[t_0,0)$. Therefore, we obtain
\begin{align*}
&\tilde{\lambda}_{\infty}(t)=\sqrt{\frac{8E_0}{\|yQ\|_2^2}}\left|t\right|(1+\epsilon_{\tilde{\lambda},0}(t)),\quad \tilde{b}_{\infty}(t)=\frac{8E_0}{\|yQ\|_2^2}\left|t\right|(1+\epsilon_{\tilde{b},0}(t)),\\
&\|\tilde{\varepsilon}_{\infty}(t)\|_{H^1}\lesssim \left|t\right|^{K+\frac{\alpha}{2}},\quad \|y\tilde{\varepsilon}_{\infty}(t)\|_2\lesssim \left|t\right|^{K+\frac{\alpha}{2}-1},\quad \left|\epsilon_{\tilde{\lambda},0}(t)\right|\lesssim |t|^M,\quad \left|\epsilon_{\tilde{b},0}(t)\right|\lesssim |t|^M
\end{align*}
from the uniform estimates in Lemma \ref{uniesti}. Consequently, we obtain Theorem \ref{Thm:exist-3}.

Finally, check the energy. We obtain
\begin{align*}
E\left(u_n\right)-E\left(P_{\tilde{\lambda}_n,\tilde{b}_n,\tilde{\gamma}_n}\right)=&\int_0^1\left\langle E'(P_{\tilde{\lambda}_n,\tilde{b}_n,\tilde{\gamma}_n}+\tau \tilde{\varepsilon}_{\tilde{\lambda}_n,\tilde{b}_n,\tilde{\gamma}_n}),\tilde{\varepsilon}_{\tilde{\lambda}_n,\tilde{b}_n,\tilde{\gamma}_n}\right\rangle d\tau\\
=&O\left(\frac{1}{{\tilde{\lambda}_n}^2}\left(\|\tilde{\varepsilon}_n\|_{H^1}+\tilde{b}_n\|y\tilde{\varepsilon}\|_2\right)\right)=o(1).
\end{align*}
Similarly,
\[
E\left(u_\infty\right)-E\left(P_{\tilde{\lambda}_\infty,\tilde{b}_\infty,\tilde{\gamma}_\infty}\right)=o(1).
\]
From continuity of energy,
\[
\lim_{n\rightarrow \infty}E\left(P_{\tilde{\lambda}_n,\tilde{b}_n,\tilde{\gamma}_n}\right)=E\left(P_{\tilde{\lambda}_\infty,\tilde{b}_\infty,\tilde{\gamma}_\infty}\right).
\]
Therefore, we obtain
\[
E\left(u_\infty\right)=E_0+o_{t\nearrow0}(1).
\]
From energy conservation, $E\left(u_\infty\right)=E_0$.
\end{proof}

\section{Proofs of Theorems \ref{Thm:propmmbs-2} and \ref{Thm:propmmbs-1}}
In this section, let $K=0$ and $K'$ be sufficiently large.

Firstly, let $u$ be a solution for \eqref{NLS} and define $w$ by
\[
w(t,x):=\overline{u}(-t,x).
\]
Then $w$ is also a solution for \eqref{NLS}. Therefore, if $u$ blows up, we may assume that $u$ blows up at a positive time $T\in(0,\infty]$.

We define $\hat{\lambda}$ and $v$ by
\[
\hat{\lambda}(t):=\frac{\|\nabla Q\|_2}{\|\nabla u(t)\|_2},\quad v(t,x):=\hat{\lambda}(t)^\frac{N}{2}u(t,\hat{\lambda}(t)x).
\]
Then
\[
\|v\|_2=\|Q\|_2,\quad \|\nabla v\|_2=\|\nabla Q\|_2,\quad \limsup_{t\rightarrow T}E_{\text{crit}}(v(t))=0
\]
hold. Therefore, there exist $\hat{x}:(0,T)\rightarrow\mathbb{R}^N$ and $\hat{\gamma}:(0,T)\rightarrow\mathbb{R}$ such that
\begin{align}
\label{predec}
\hat{\lambda}(t)^\frac{N}{2}u(t,\hat{\lambda}(t)(x-\hat{x}(t)))e^{i\hat{\gamma}(t)}\rightarrow Q\quad \text{in }H^1\quad(t\rightarrow T)
\end{align}
(e.g., see \cite{MRUPB}). Assuming that $N\geq 2$ and $u$ is radial, we obtain $\hat{x}=0$.

From Lemma \ref{decomposition}, for any $\epsilon_0>0$ there exists $t_0$ that is sufficiently close to $T$ such that we obtain the decomposition of $u$ on $(t_0,T)$: 
\[
u(t,x)=\frac{1}{\tilde{\lambda}(t)^{\frac{N}{2}}}\left(P+\tilde{\varepsilon}\right)\left(t,\frac{x}{\tilde{\lambda}(t)}\right)e^{-i\frac{\tilde{b}(t)}{4}\frac{|x|^2}{\tilde{\lambda}(t)^2}+i\tilde{\gamma}(t)}.
\]
In particular, from arbitrariness of $\epsilon_0$ and uniqueness of the decomposition in Lemma \ref{decomposition},
\begin{align}
\label{paramconv}
\lim_{t\rightarrow T}\frac{\tilde{\lambda}(t)}{\hat{\lambda}(t)}=1,\quad \lim_{t\rightarrow T}\tilde{b}(t)=0,\quad \lim_{t\rightarrow T}e^{i(\hat{\gamma}(t)+\tilde{\gamma}(t))}=1
\end{align}
hold.

Let $\hat{\varepsilon}$ be defined by
\[
\hat{\varepsilon}(t,y):=\tilde{\varepsilon}\left(t,y\right)e^{-i\frac{\tilde{b}(t)|y|^2}{4}}.
\]
Then, from \eqref{predec} and \eqref{paramconv},
\[
\lim_{t\rightarrow T}\|\hat{\varepsilon}\|_{H^1}=0
\]
holds. Moreover, from mass conversation,
\[
-(Q,\tilde{\varepsilon})_2=\frac{1}{2}\|\hat{\varepsilon}\|_2^2+O\left(\tilde{\lambda}^\alpha\left(|\tilde{b}|^2+\lambda^\alpha\right)\right)+O(\tilde{\lambda}^\alpha\|\hat{\varepsilon}\|_2).
\]
Therefore, from energy conversation, we obtain
\begin{align*}
\tilde{\lambda}^2E(u_0)&=\frac{\tilde{b}^2}{8}\|yQ\|_2^2+\frac{1}{2}\|\nabla\hat{\varepsilon}\|_2^2+\frac{1}{2}\|\hat{\varepsilon}\|_2^2-\frac{1}{2}\int_{\mathbb{R}^N}d^2F(Q)(\hat{\varepsilon},\hat{\varepsilon})dy-\frac{C_1\tilde{\lambda}^\alpha}{p+1}\|Q\|_{p+1}^{p+1}-\frac{C_2\tilde{\lambda}^\alpha}{2}\||\cdot|^{-\sigma}Q\|_2^2\\
&\hspace{20pt}+O(\tilde{\lambda}^{2\alpha}+|b|\tilde{\lambda}^\alpha+\tilde{\lambda}^\alpha\|\hat{\varepsilon}\|_{H^1})+o(\tilde{\lambda}^{2\alpha}+\|\hat{\varepsilon}\|_{H^1}^2).
\end{align*}
Accordingly, from \eqref{Lcoer},
\[
\tilde{\lambda}^2E(u_0)+\frac{C_1\tilde{\lambda}^\alpha}{p+1}\|Q\|_{p+1}^{p+1}+\frac{C_2\tilde{\lambda}^\alpha}{2}\||\cdot|^{-\sigma}Q\|_2^2\gtrsim \tilde{b}^2+\|\hat{\varepsilon}\|_{H^1}^2
\]
holds. Consequently, we obtain the following lemma:

\begin{lemma}
\label{bepsiesti}
If $(C_1,C_2)=(\pm \omega,\mp 1)$, then
\[
\tilde{\lambda}^2E(u_0)\gtrsim \tilde{b}^2+\|\hat{\varepsilon}\|_{H^1}^2.
\]
If $(C_1,C_2)=(C_0,-1)$ with $C_0>\omega$ or $(C_1,C_2)=(-C_0,1)$ with $0<C_0<\omega$, then
\[
\tilde{\lambda}^\alpha\gtrsim \tilde{b}^2+\|\hat{\varepsilon}\|_{H^1}^2.
\]
\end{lemma}

\begin{corollary}
If $(C_1,C_2)=(\pm \omega,\mp 1)$ and $u$ is a radial blow-up solution for \eqref{NLS} with critical mass, then
\[
E(u)>0.
\]
\end{corollary}

\begin{proof}
From Lemma \ref{bepsiesti}, $E(u)\geq 0$ is obvious.

We assume $E(u)=0$. Then $\tilde{b}=0$ and $\tilde{\varepsilon}=0$. Namely, from \eqref{Pdef},
\[
u(t,x)=\frac{1}{\tilde{\lambda}(t)^{\frac{N}{2}}}\left(Q\left(\frac{y}{\tilde{\lambda}(t)}\right)+\sum_{k=0}^{K'}\tilde{\lambda}^{(k+1)\alpha}P_{0,k}^+\left(\frac{x}{\tilde{\lambda(t)}}\right)\right)\left(t,\frac{x}{\tilde{\lambda}(t)}\right)e^{i\tilde{\gamma}(t)}.
\]
Since $u$ is a solution for \eqref{NLS2} with $C_0=\omega$, we obtain
\begin{align*}
0&=-i\tilde{\lambda}\frac{\partial\tilde{\lambda}}{\partial t}\Lambda\left(Q+\tilde{\lambda}^\alpha Z\right)+i\sum_{k=0}^{K'}(k+1)\alpha\tilde{\lambda}^{(k+1)\alpha+1}\frac{\partial\tilde{\lambda}}{\partial t}P_{0,k}^+-\tilde{\lambda}^2\frac{\partial\tilde{\gamma}}{\partial t}\left(Q+\tilde{\lambda}^\alpha Z\right)\\
&\hspace{20pt}+\Delta \left(Q+\tilde{\lambda}^\alpha Z\right)+|Q+\tilde{\lambda}^\alpha Z|^{\frac{4}{N}}\left(Q+\tilde{\lambda}^\alpha Z\right)\pm \omega\tilde{\lambda}^\alpha|Q+\tilde{\lambda}^\alpha Z|^{p-1}\left(Q+\tilde{\lambda}^\alpha Z\right)\mp\frac{\tilde{\lambda}^\alpha}{|y|^{2\sigma}}\left(Q+\tilde{\lambda}^\alpha Z\right).
\end{align*}
Therefore, taking imaginary part, we obtain
\[
0=\frac{\partial\tilde{\lambda}}{\partial t}\left(\Lambda Q+\tilde{\lambda}^\alpha \left(\Lambda Z+\sum_{k=1}^{K'}(k+1)\alpha\tilde{\lambda}^{(k+1)\alpha}P_{0,k}^+\right)\right).
\]
Since $\Lambda Q\neq 0$ and $\tilde{\lambda}\rightarrow 0$ as $t\rightarrow T$, $0=\frac{\partial\tilde{\lambda}}{\partial t}$. It means that $\tilde{\lambda}$ is a constant. However, it contradicts $\tilde{\lambda}>0$ and $\tilde{\lambda}\rightarrow 0$ as $t\rightarrow T$. Consequently, $E(u)\neq 0$.
\end{proof}

\begin{proof}[Proofs of Theorems \ref{Thm:propmmbs-2} and \ref{Thm:propmmbs-1}]
From Lemma \ref{Modesti-1},
\[
\left|\tilde{\lambda}\frac{\partial\tilde{\lambda}}{\partial t}+\tilde{b}\right|\lesssim \|\hat{\varepsilon}\|_{H^1}^2+\tilde{\lambda}^{2\alpha}
\]
holds on $(t_0,T)$.

We assume $(C_1,C_2)=(C_0,-1)$ with $C_0>\omega$ or $(C_1,C_2)=(-C_0,1)$ with $0<C_0<\omega$. Then, since $|\tilde{b}|\lesssim \tilde{\lambda}^\frac{\alpha}{2}$ from Lemma \ref{bepsiesti}, we obtain
\[
\left|\tilde{\lambda}\frac{\partial\tilde{\lambda}}{\partial t}\right|\lesssim \tilde{\lambda}^{\frac{\alpha}{2}}.
\]
Therefore,
\[
1\gtrsim\left|\frac{\partial}{\partial t}\tilde{\lambda}^{2-\frac{\alpha}{2}}\right|
\]
holds. Integrating on $(t,T)$, we obtain
\[
\tilde{\lambda}(t)^{2-\frac{\alpha}{2}}\lesssim T-t.
\]
Consequently,
\[
\|\nabla u(t)\|_2\sim\frac{1}{\tilde{\lambda}(t)}\gtrsim \frac{1}{(T-t)^\frac{2}{4-\alpha}}.
\]

The same can be proved when assuming $(C_1,C_2)=(\pm \omega,\mp 1)$.
\end{proof}

\end{document}